\RequirePackage{fix-cm}
\documentclass[smallextended]{svjour3}       
\smartqed  
\usepackage{graphicx}
\usepackage[numbers]{natbib}
\usepackage{amsmath, amssymb}

\journalname{Communications in Mathematics and Statistics}
\begin{document}

\title{Regularity Properties for Sparse Regression\thanks{Fan's research was partially supported by NIH grants R01GM100474-04 and
NIH R01-GM072611-10 and NSF grants DMS-1206464 and DMS-1406266.  The bulk of the research was carried out while Edgar Dobriban was an undergraduate student at Princeton University.}
}


\author{Edgar Dobriban        \and
        Jianqing Fan 
}


\institute{E. Dobriban \at
              Department of Statistics, Stanford University \\
              \email{dobriban@stanford.edu}           
           \and
           J. Fan \at
             Department of Operations Research and Financial Engineering, Princeton University \\
              \email{jqfan@princeton.edu}
}

\date{Received: date / Accepted: date}

\maketitle

\begin{abstract}
Statistical and machine learning theory has developed several conditions ensuring that popular estimators such as the Lasso or the Dantzig selector perform well in high-dimensional sparse regression, including the restricted eigenvalue, compatibility, and $\ell_q$ sensitivity properties. However, some of the central aspects of these conditions are not well understood. For instance, it is unknown if these conditions can be checked efficiently on any given data set. This is problematic, because they are at the core of the theory of sparse regression.

Here we provide a rigorous proof that these conditions are NP-hard to check.
This shows that the conditions are computationally infeasible to verify, and raises some questions about their practical applications.

However, by taking an average-case perspective instead of the worst-case view of NP-hardness, we show that a particular condition, $\ell_q$ sensitivity, has certain desirable properties. This condition is weaker and more general than the others. We show that it holds with high probability in models where the parent population is well behaved, and that it is robust to certain data processing steps. These results are desirable, as they provide guidance about when the condition, and more generally the theory of sparse regression, may be relevant in the analysis of high-dimensional correlated observational data.

\end{abstract}

\section{Introduction}
\label{intro}

\subsection{Prologue}

Open up any recent paper on sparse linear regression -- the model $Y = X\beta + \varepsilon $, where $X$ is an $n \times p$ matrix of features, $n \ll p$, and most coordinates of $\beta$ are zero  -- and you are likely to find that the main result is of the form:
\emph{
``If the data matrix $X$ has the restricted eigenvalue/compatibility/$\ell_q$ sensitivity property, then our method will successfully estimate the unknown sparse parameter $\beta$, if the sample size is at least \ldots''
}

In addition to the sparsity of the parameter, the key condition here is the \emph{regularity} of the matrix of features, such as restricted eigenvalue/ compatibility/ $\ell_q$ sensitivity.  It states that \emph{every suitable submatrix of the feature matrix $X$ is ``nearly orthogonal''}. Such a property is crucial for the success of popular estimators like the Lasso and Dantzig selector. However, these conditions are somewhat poorly understood. For instance,  as the conditions are combinatorial, it is not known how to check them efficiently --  in polynomial time -- on any given data matrix. Without this knowledge, it is difficult to see whether or not the whole framework is relevant to any particular data analysis setting.

In this paper we seek a better understanding of these problems. We first establish that the most popular conditions for sparse regression -- restricted eigenvalue/compatibility/$\ell_q$ sensitivity -- are all {\sf NP}-hard to check. This implies that there is likely \emph{no efficient way to verify them} for deterministic matrices, and raises some questions about their practical applications.

Next, we move away from the worst-case analysis entailed by {\sf NP}-hardness, and consider an average-case, non-adversarial analysis. We show that the weakest of these conditions, $\ell_q$ sensitivity, has some desirable properties, including that it holds with high probability in well-behaved random design models, and that it is preserved under certain data processing operations.

\subsection{Formal introduction}

We now turn to a more formal and thorough introduction. The context of this paper is that high-dimensional data analysis is becoming commonplace in statistics and machine learning. Recent research shows that estimation of high-dimensional parameters may be possible if they are suitably sparse. For instance, in linear regression where most of the regression coefficients are zero, popular estimators such as the Lasso \citep{chen_atomic_2001, tibshirani_regression_1996}, SCAD
\citep{fan_li01}, and the Dantzig selector \citep{candes_dantzig_2007} can have small estimation error -- as long as the matrix of covariates is sufficiently ``regular''.

There is a large number of suitable regularity conditions, starting with the incoherence condition of Donoho and Huo \cite{donoho_uncertainty_2001}, followed by more sophisticated properties such as Candes and Tao's restricted isometry property (``RIP'') \cite{candes_decoding_2005}, Bickel, Ritov and Tsybakov's weaker and more general restricted eigenvalue (RE) condition \cite{bickel_simultaneous_2009},  and Gautier and Tsybakov's even more general $\ell_q$ sensitivity properties \cite{gautier_high-dimensional_2011}, which also apply to instrumental variables regression.

While it is known that these properties lead to desirable guarantees on the performance of popular statistical methods, it is largely unknown whether they hold in practice. Even more, it is not known how to efficiently check if they hold for any given data set. Due to their combinatorial nature, it is thought that they may be computationally hard to verify  \citep{tao_open_2007, raskutti_restricted_2010, daspremont_testing_2011}. The assumed difficulty of the computation has motivated convex relaxations for approximating the restricted isometry constant \citep{daspremont_optimal_2008, lee_computing_2008} and $\ell_q$ sensitivity \citep{gautier_high-dimensional_2011}.

However, a rigorous proof is missing. A proof would be desirable for several reasons: (1) to show definitively that there is no computational ``shortcut'' to find their values, (2) to increase our understanding of \emph{why} these conditions are difficult to check, and therefore (3) to guide the development of the future theory of sparse regression, based instead on efficiently verifiable conditions.

In this paper we provide such a proof. We show that checking any of the restricted eigenvalue, compatibility, and $\ell_q$ sensitivity properties for general data matrices is $\mathsf{NP}$-hard (Theorem \ref{hardness}). This implies that there is no polynomial-time algorithm to verify them, under the widely believed assumption that $\mathsf{P} \ne \mathsf{NP}$. This raises some questions about the relevance of these conditions 
to practical data analysis.

We do not attempt to give a definitive answer here, and instead provide some positive results to enhance our understanding of these conditions. While the previous {\sf NP}-hardness analysis referred to a worst-case scenario, we next take an average-case, non-adversarial perspective. Previous authors studied RIP, RE and compatibility from this perspective, as well as the relations between these conditions \citep{van_de_geer_conditions_2009}. We study $\ell_q$ sensitivity, for two reasons:  First, it is more general than other regularity properties in terms of the correlation structures it can capture, and thus potentially applicable to more highly correlated data. Second, it applies not just to ordinary linear regression, but also to instrumental variables regression, which is relevant in applications such as economics.

Finding conditions under which $\ell_q$ sensitivity holds is valuable for several reasons: (1) since it is hard to check the condition computationally on any given data set, it is desirable to have some other way to ascertain it, even if that method is somewhat speculative, and (2) it helps us to compare the situations -- and statistical models -- where this condition is most suitable to the cases where the other conditions are applicable, and thus better understand its scope.

Hence, to increase our understanding of when $\ell_q$ sensitivity may be relevant, we perform a probabilistic -- or ``average case'' -- analysis, and consider a model where the data is randomly sampled from suitable distributions. In this case, we show that there is a natural ``population'' condition which is sufficient to ensure that $\ell_q$ sensitivity holds with high probability (Theorem \ref{population_sample}). This complements 
the results for RIP \cite[e.g.,][]{rauhut_compressed_2008,vershynin_introduction_2010}, and RE \citep{raskutti_restricted_2010, rudelson_reconstruction_2012}. 
Further, we define  an explict \emph{k-comprehensive} property (Definition \ref{k-comp-def}) which implies $\ell_1$ sensitivity (Theorem \ref{k-comprehensive}). Such a condition is of interest because there are very few explicit examples where one can ascertain that $\ell_q$ sensitivity holds.

Finally, we show that the $\ell_q$ sensitivity property is preserved under several data processing steps that may be used in practice (Proposition \ref{linear_operations}). This shows that, while it is initially hard to ascertain this property, it may be somewhat robust to downstream data processing.

We introduce the problem in Section \ref{definitions_setup}. Then, in Section \ref{mainresult} we present our results, with a discussion in Section \ref{discussion}, and provide the proofs in Section \ref{proofs}.

\section{Setup}
\label{definitions_setup}

We introduce the problems and properties studied, followed by some notions from computational complexity.

\subsection{Regression problems and estimators}

Consider the linear model $Y  = X \beta + \varepsilon$, where $Y$ is an $n \times 1$ response vector, $X$ is an $n \times p$ matrix of $p$ covariates, $\beta$ is a $p \times 1$ vector of coefficients, and $\varepsilon$ is an $n \times 1$ noise vector of independent $N(0,\sigma^2)$ entries. The observables are $Y$ and $X$, where $X$ may be deterministic or random, and we want to estimate the fixed unknown $\beta$. Below we will briefly present the modeling and the estimation procedures that are required, while for the full details we refer to the original publications.

In the case when $n < p$, it is common to assume sparsity, viz., most of the coordinates of $\beta$ are zero. We do not know the locations of nonzero coordinates.  A popular estimator in this case is the Lasso \citep{tibshirani_regression_1996, chen_atomic_2001}, which for a given regularization parameter $\lambda$ solves the optimization problem:
$$
 \hat\beta_{\mbox{\scriptsize Lasso}} = \arg\min_{\beta} \frac{1}{2n}|Y - X\beta|_2^2 + \lambda\sum_{i=1}^{p}|\beta_i|,
$$

The Dantzig selector is another estimator for this problem, which for a known noise level $\sigma$, and with a tuning parameter $A$, takes the form \citep{candes_dantzig_2007}:
$$
\hat\beta_{\mbox{\scriptsize Dantzig}} = \arg\min{|\beta|_1}, \mbox{ subject to } \left| \frac1n  X^T (Y - X\beta)\right|_\infty \leq \sigma A \sqrt \frac {2 \log(p)}{n}.
$$
See \cite{fan14} for a view from the sparest solution in high-confidence set and its generalizations.

In instrumental variables regression we start with the same linear model $y = \sum_{i=1}^{p}x_i\beta_i + \varepsilon$. Now some covariates $x_i$ may be correlated with the noise $\varepsilon$, in which case they are called endogenous. Further,  we have additional variables $z_i$, $i=1, \ldots L$, called instruments, that are uncorrelated with the noise. In addition to $X$, we observe $n$ independent samples of $z_i$, which are arranged in the $n \times L$ matrix $Z$. In this setting, \cite{gautier_high-dimensional_2011} propose the Self-Tuning Instrumental Variables (STIV) estimator, a generalization of the Dantzig selector, which solves the optimization problem:
\begin{equation}
    \label{stiv}
    \min_{(\beta,\sigma) \in \mathcal{I}}( |D_X^{-1}\beta|_1 + c\sigma),
\end{equation}
with the minimum over the polytope
$
    \mathcal{I} = \{ (\beta,\sigma) \in \mathbb{R}^{p+1}:  n^{-1}|   D_Z Z^T (Y - X\beta)|_\infty \leq \sigma A \sqrt {2 \log(L)/n}$, $Q(\beta) \le \sigma^2 \}.
$
Here $D_X$ and $D_Z$ are diagonal matrices with $(D_X)_{ii}^{-1} = \max_{k=1,\ldots,n}|x_{ki}|$, $(D_Z)_{ii}^{-1} = \max_{k=1,\ldots,n}|z_{ki}|$, $Q(\beta) = n^{-1} |Y-X\beta|_2^2$, and $c$ is a constant whose choice is described in \cite{gautier_high-dimensional_2011}.  When $X$ is exogenous, we can take $Z = X$, which reduces to Dantzig type of selector.

\subsection{Regularity properties}

The performance of the above estimators is characterized under certain ``regularity properties''. These depend on the union of cones $C(s, \alpha)$ -- called ``the cone'' for brevity -- which is the set of vectors such that the $\ell_1$ norm is concentrated on some $s$ coordinates:
$$
 C(s,\alpha) = \{ v \in \mathbb{R}^p:\exists S \subset \{1, \ldots, p\}, |S|=s, \alpha |v_S|_1 \geq |v_{S^c}|_1 \},
$$
where $v_A$ is the subvector of $v$ with the entries from the subset $A$.

The properties discussed here depend on a triplet of parameters $(s, \alpha, \gamma)$, where $s$ is the sparsity size of the problem, $\alpha$ is the cone opening parameter in $C(s, \alpha)$, and $\gamma$ is the lower bound. First, the Restricted Eigenvalue condition $RE(s, \alpha, \gamma)$ from \cite{bickel_simultaneous_2009, koltchinskii_dantzig_2009} holds for a fixed matrix $X$ if
$$
    \frac{|Xv|_2}{|v_S|_2} \geq \gamma, \mbox{ for all } v \in C(s, \alpha), \alpha |v_S|_1 \geq |v_{S^c}|_1.
$$

We emphasize that this property, and the ones below, are defined for arbitrary deterministic matrices -- but later we will consider them for randomly sampled data. \cite{bickel_simultaneous_2009} shows that if the normalized data matrix ${n}^{-1/2}X$  obeys $RE(s, \alpha, \gamma)$ and $\beta$ is $s$-sparse, then the estimation error is small in the sense that $|\hat{\beta} - \beta|_2  = O_P\left({\gamma^{-2}} \sqrt {s\log p/n} \right)$ and $|\hat{\beta} - \beta|_1  = O_P\left({\gamma^{-2}} s \sqrt {\log p/n} \right)$,
for both the Dantzig and Lasso selectors.  See \cite{fan14} for more general results and simpler arguments. The ``cone opening'' $\alpha$ required in the restricted eigenvalue property equals 1 for the Dantzig selector, and 3 for the Lasso.

Next, the determinstic matrix
$X$ obeys the \bf compatibility \rm condition with positive parameters $(s, \alpha, \gamma)$  \citep{van_de_geer_deterministic_2007},  if
$$
  \frac{\sqrt{s}|Xv|_2}{|v_S|_1} \geq \gamma, \mbox{ for all } v \in C(s, \alpha), \alpha |v_S|_1 \geq |v_{S^c}|_1.
$$

The two conditions are very similar. The only difference is the change from $\ell_2$  to $\ell_1$ norm in the denominator. The inequality $|v_S|_1 \leq \sqrt{s} |v_S|_2 $ shows that the compatibility conditions are -- formally at least -- weaker than the RE assumptions. \cite{van_de_geer_deterministic_2007} provides an $\ell_1$ oracle inequality for the Lasso under the compatibility condition, see also \cite{van_de_geer_conditions_2009, buhlmann_statistics_2011}.

Finally, for $q \ge 1$,  the deterministic matrices $X$ of size $n \times p$ and $Z$  of size $n \times L$ satisfy the \bf $\ell_q$ sensitivity \rm property with parameters $(s, \alpha, \gamma)$, if

$$\frac{s^{1/q}|n^{-1} Z^TXv|_{\infty}}{|v|_q} \geq \gamma, \mbox{ for all } v \in C(s, \alpha).
$$

If $Z = X$, the definition is similar to the cone invertibility factors \cite{ye_rate_2010}.
\cite{gautier_high-dimensional_2011} shows that $\ell_q$ sensitivity is weaker than the RE and compatibility conditions, meaning that in the special case when $Z=X$,  the RE property of $X$ implies the $\ell_q$ sensitivity of $X$. We note that the definition in \cite{gautier_high-dimensional_2011} differs in normalization, but that is not essential. The details are that we have an additional $s^{1/q}$ factor (this is to ensure direct comparability to the other conditions), and we do not normalize by the diagonal matrices $D_X,D_Z$ for simplicity (to avoid the dependencies introduced by this process). One can easily show that the un-normalized $\ell_q$ condition is sufficient for the good performance of an un-normalized version of the STIV estimator.

Finally, we introduce incoherence and the restricted isometry property, which are not analyzed in this paper, but are instead used for illustration purposes. For a deterministic $n \times p$ matrix $X$ whose columns $\{X_j\}_{j=1}^p$ are normalized to length $\sqrt{n}$, the mutual incoherence condition holds if $X_i^T X_j \leq \gamma/s$ for some positive $\gamma$. Such a notion was defined in \cite{donoho_uncertainty_2001}, and later used by  \cite{bunea_sparsity_2007} to derive oracle inequalities for the Lasso.

A deterministic matrix $X$ obeys the restricted isometry property with parameters $s$ and $\delta$ if $(1-\delta)|v|_2^2$ $\le |Xv|_2^2 \le$ $(1+\delta)|v|_2^2$ for all $s$-sparse vectors $v$ \citep{candes_decoding_2005}.

\subsection{Notions from computational complexity}

To state formally that the regularity conditions are hard to verify, we need some basic notions from computational complexity theory. Here problems are classified according to the computational resources -- such as time and memory -- needed to solve them \citep{arora_computational_2009}. A well-known complexity class is $\mathsf{P}$, consisting of the problems decidable in polynomial time in the size of the input. For input encoded in $n$ bits, a \emph{yes} or \emph{no} answer must be found in time $O(n^k)$ for some fixed $k$. A larger class is $\mathsf{NP}$, the decision problems for which already existing solutions can be verified in polynomial time. This is usually much easier than solving the question itself in polynomial time. For instance, the subset-sum problem: ``Given an input set of integers, does there exist a subset with zero sum?'' is in $\mathsf{NP}$, since one can easily check a candidate solution -- i.e., a subset of the given integers -- to see if it indeed sums to zero. However, finding this subset seems harder, as simply enumerating all subsets is not a polynomial-time algorithm.

Formally, the definition of $\mathsf{NP}$ requires that if the answer is \emph{yes}, then there exists an easily verifiable proof.  We have $\mathsf{P} \subset \mathsf{NP}$, since a polynomial-time solution is a certificate verifiable in polynomial time.  However, it is a famous open problem to decide if $\mathsf{P}$ equals $\mathsf{NP}$ \citep{cook_p_2000}. It is widely believed in the complexity community that $\mathsf{P} \neq \mathsf{NP}$.

To compare the computational hardness of various problems, one can reduce known hard problems to the novel questions of interest, thereby demonstrating the difficulty of the novel problems. Specifically, a problem $A$ is polynomial-time reducible to a problem $B$, if an oracle solving $B$ -- that is, an immediate solver for an instance of $B$ -- can be queried once to give a polynomial-time algorithm to solve $A$. This is also known as a polynomial-time many-one reduction, strong reduction, or Karp reduction. A problem is \emph{$\mathsf{NP}$-hard}  if every problem in $\mathsf{NP}$ reduces to it, namely it is at least as difficult as all other problems in $\mathsf{NP}$. If one reduces a known $\mathsf{NP}$-hard problem to a new question, this demonstrates the  $\mathsf{NP}$-hardness of the new problem.

If indeed $\mathsf{P} \neq \mathsf{NP}$, then there are no polynomial time algorithms for \emph{$\mathsf{NP}$}-hard problems, implying that these are indeed computationally difficult.

\section{Results}
\label{mainresult}

\subsection{Computational Complexity}

We now show that the common conditions needed for successful sparse estimation are unfortunately $\mathsf{NP}$-hard to verify. These conditions appear prominently in the theory of high-dimensional statistics, large-scale machine learning, and compressed sensing. In compressed sensing, one can often choose, or ``engineer'', the matrix of covariates such that it is as regular as possible -- choosing for instance a matrix with iid Gaussian entries. It is well known that the restricted isometry property and its cousins will then hold with high probability.

In contrast, in statistics and machine learning, the data matrix is often observational -- or ``given to us'' -- in the application. In this case, it is not known a priori whether the matrix is regular, and one may be tempted to try and verify it. Unfortunately, our results show that this is hard. This distinction between compressed sensing and statistical data analysis was the main motivation for us to write this paper, after the computational difficulty of verifying the restricted isometry property has been established in the information theory literature 
\citep{bandeira_certifying_2013}. We think that researchers in high-dimensional statistics will benefit from the broader view which shows that not just RIP, but also RE, $\ell_q$ sensitivity, etc., are hard to check. Formally:

\begin{theorem}
\label{hardness}
Let $X$ be an $n \times p$ matrix, $Z$ an $n \times L$ matrix, $0 < s <n$, and $\alpha, \gamma>0$. It is $\mathsf{NP}$-hard to decide any of the following problems:
\begin{enumerate}
\item Does $X$ obey the restricted eigenvalue condition with parameters $(s, \alpha, \gamma)$?
\item Does $X$ satisfy the compatibility conditions with parameters $(s, \alpha, \gamma)$?
\item Does $(X, Z)$ have the $\ell_q$ sensitivity property with parameters $(s, \alpha, \gamma)$?
\end{enumerate}
\end{theorem}

The proof of Theorem \ref{hardness} is relegated to Section \ref{hardnessproof}, and  builds on the recent results that computing the spark and checking restricted isometry are $\mathsf{NP}$-hard \citep{bandeira_certifying_2013, tillmann_computational_2012}.

\subsection{$\ell_q$ sensitivity for correlated designs}

Since it is hard to check the properties \emph{in the worst case} on a generic data matrix, it may be interesting to know that they hold at least \emph{under certain conditions}. To understand when this may occur, we consider probabilistic models for the data, which amounts to an \emph{average case analysis}. This type of analysis is common in statistics. To this end, we first need to define a ``population'' version of $\ell_q$ sensitivity that refers to the parent population from which the data is sampled. Let $\underline{X}$ and $\underline{Z}$ be $p$- and $L$-dimensional zero-mean random vectors and denote by $\Psi = \mathbb{E}$$ \underline{Z} \underline{X}^T$ the $L \times p$ matrix of covariances with $ \Psi_{ij} = \mathbb{E}(Z_iX_j)$. We say that $\Psi$ satisfies the $\ell_q$ sensitivity property with parameters $(s, \alpha, \gamma)$ if $ \min_{v \in C(s, \alpha)} s^{1/q}\left|\Psi v\right|_\infty/|v|_q \geq \gamma$. One sees that we simply replaced $n^{-1}\underline{Z} \underline{X}^T$ from the original definition with its expectation, $\Psi$.

It is then expected that for sufficiently large samples, random matrices with rows sampled independently from a population with the $\ell_q$ sensitivity property will inherit this condition. However, it is non-trivial to understand the required sample size, and its dependence on the moments of the random quantities. To state precisely the required probabilistic assumptions, we recall that the sub-gaussian norm of a random variable is defined as $\| X \|_{\psi_2} = \sup_{p \geq 1} p^{-1/2}(\mathbb{E}|X|^p)^{1/p}$ \citep[see e.g.,][]{vershynin_introduction_2010}.  The sub-gaussian norm (or sub-gaussian constant) of a $p$-dimensional random vector $\underline{X}$ is then defined as
$ \| \underline{X} \|_{\psi_2} = \sup_{x: \|x\|_2=1} \| \langle \underline{X}, x \rangle \|_{\psi_2}$.

Our result establishes sufficient conditions for $\ell_q$ sensitivity to hold for random matrices, under three broad conditions including sub-gaussianity:
\begin{theorem}
\label{population_sample}

 Let $\underline{X}$ and $\underline{Z}$ be zero-mean random vectors, such that the matrix of population covariances $\Psi$ satisfies the $\ell_q$ sensitivity property with parameters $(s, \alpha, \gamma)$. Given $n$ iid samples and any $a,\delta>0$, the matrix $\hat\Psi = n^{-1}Z^TX$ has the $\ell_q$ sensitivity property with parameters $(s, \alpha, \gamma- \delta)$,  with high probability, in each of the following settings:

\begin{enumerate}
\item If $\underline{X}$ and $\underline{Z}$ are sub-gaussian with fixed constants, then sample $\ell_q$ sensitivity holds with probability at least $1 - (2pL)^{-a}$, provided that the sample size is at least $ n \geq c s^2 \log(2pL)$.
\item If the entries of the vectors are bounded by fixed constants, the same statement holds.
\item If  the entries have bounded moments: $\mathbb{E}|X_i|^{4r} <$ $C_x$ $< \infty$, $\mathbb{E}|Z_j|^{4r} < $ $C_z$ $< \infty$ for some positive integer $r$ and all $i$, $j$, then the $\ell_q$ sensitivity property holds with probability at least $1  - 1/n^a$, assuming the sample size is at least $ n^{1 - a/r} \geq c s^2 (pL)^{1/r}$.
\end{enumerate}
\end{theorem}

The constant $c$ does not depend on $n,L,p$ and $s$, and it is given in the proofs in Section \ref{proof_popsample}.

The general statement of the theorem is applicable to the specific case where $\underline{Z} = \underline{X}$. 
Related results have been obtained for the RIP \citep{rauhut_compressed_2008, rudelson_reconstruction_2012} and RE conditions \citep{raskutti_restricted_2010, rudelson_reconstruction_2012}. Our results complement theirs for a weaker notion of $\ell_q$ sensitivity property.

Next, we aim to achieve a better understanding of the population $\ell_q$ sensitivity property by giving some explicit sufficient conditions where it holds. Modeling covariance matrices in high dimensions is challenging, as there are few known explicit models. For instance, the examples given in \cite{raskutti_restricted_2010} to illustrate RE are quite limited, and include only diagonal, diagonal plus rank one, and ARMA covariance matrices. Therefore we think that the explicit conditions below are of interest, even if they are somewhat abstract.

We start from the case when $\underline{Z} = \underline{X}$, in which case $\Psi$ is the covariance matrix of $\underline{X}$.  In particular, if $\Psi$ equals the identity matrix $I_p$ or nearly the identity, then $\Psi$ is $\ell_q$-sensitive. Inspired by this diagonal case, we introduce a more general condition.

\begin{definition}
\label{k-comp-def}
The $L \times p$ matrix $\Psi$ is called \emph{s-comprehensive} if for any subset $S \subset \{1,\ldots,p\}$ of size $s$, and for each pattern of signs $\varepsilon \in \{-1,1\}^S$, there exists either a row $w$ of  $\Psi$ such that $\emph{sgn}(w_i) = \varepsilon_i$ for $i \in S$, and $w_i = 0$ otherwise, or a row with $\emph{sgn}(w_i) = -\varepsilon_i$ for $i \in S$, and $w_i = 0$ otherwise.
\end{definition}

In particular, when $L=p$,  diagonal matrices with nonzero diagonal entries are 1-comprehensive. 
More generally, when $L \ne p$, we have by simple counting the inequality  $L \ge 2^{s-1} {p \choose s}$, which shows that the number of instruments $L$ must be large for the $s$-comprehensive property to be applicable. In problems where there are many potential instruments, this may be reasonable. To go back to our main point, we show that an $s$-comprehensive covariance matrix is $\ell_1$-sensitive.

\begin{theorem}
\label{k-comprehensive}
Suppose the $L \times p$ matrix of covariances $\Psi$ is s-comprehensive, and that all nonzero entries in $\Psi$ have absolute value at least $c>0$. Then $\Psi$ obeys the $\ell_1$ sensitivity property with parameters $s,\alpha$ and $\gamma = s c /(1+\alpha)$.
\end{theorem}

The proof of Theorem \ref{k-comprehensive} is found in Section \ref{proof_comprehensive}. The theorem presents a trade-off between the number of instruments $L$ and their strength, by showing that with a large subset size $s$ -- and thus $L$ -- a smaller minimum strength $c$ is required to achieve the same $\ell_1$ sensitivity lower bound $\gamma$.

Finally, to improve our understanding of the relationship between the various conditions, we now give several examples. They show that $\ell_q$ sensitivity is more general than the rest.  The proofs of the following claims can be found in Section \ref{proof_example}.

\begin{example}
\label{covariance_example_1}
 If $\Sigma$ is a diagonal matrix with entries $d_1, d_2, \ldots, d_p$, then the restricted isometry property holds if $1 + \delta \ge d_i \ge 1 - \delta$ for all $i$. Restricted eigenvalue only requires $d_i \ge \gamma$; the same is required for compatibility. This example shows why restricted isometry is the most stringent requirement. Further, $\ell_1$ sensitivity holds even if a finite number of $d_i$ go to zero at rate $1/s$. In this case, all other regularity conditions fail.  This is an example where $l_q$ regularity holds under broader conditions than the others.
\end{example}

The next examples further delineate between the various properties.

\begin{example}
For the equal correlations model $\Sigma  = (1-\rho) I_p + \rho e e^T$, with $e = (1,\ldots,1)^T$, restricted isometry requires $\rho < 1/(s-1)$. In contrast, restricted eigenvalue, compatibility, and $\ell_q$ sensitivity hold for any $\rho$, and the resulting lower bound $\gamma$ is $1-\rho$ (see  \cite{van_de_geer_conditions_2009, raskutti_restricted_2010}).
\end{example}

\begin{example}
\label{covariance_example_2}
If $\Sigma$ has diagonal entries equal to 1, $\sigma_{12}=\sigma_{21} = \rho$, and all other entries are equal to zero, then compatibility and $\ell_1$ sensitivity hold as long as $1-\rho \asymp 1/s$ (Section \ref{proof_example}). In such a case, however, the restricted eigenvalues are of order $1/s$. This is an example where compatibility and $\ell_1$ sensitivity hold but the restricted eigenvalue condition fails.
\end{example}

\subsection{Operations preserving regularity}
\label{operations}

In data analysis, one often processes data by normalization or feature merging. Normalization is performed to bring variables to the same scale. Features are merged via sparse linear combinations to reduce dimension and avoid multicollinearity.  Our final result shows that $\ell_q$ sensitivity is preserved under the above operations, and even more general ones. This may be of interest in cases where downstream data processing is performed after an initial step where the regularity conditions are ascertained. 
Let $X$ and $Z$ be as above. First, note that the $\ell_q$ sensitivity only depends on the inner products $ZX^T$, therefore it is preserved under simultaneous orthogonal transformations on each covariate $X' = MX$, $Z' = MZ$ for any orthogonal matrix $M$. The next result defines broader classes of transformations that preserve $\ell_q$ sensitivity. Admittedly the transformations we consider are abstract, but they include some concrete examples, and represent a simple first step to understanding what kind of data processing steps are ``admissible'' and do not destroy regularity. Furthermore, the result is very elementary, but the goal here is not technical sophistication, but rather increasing our understanding of the behavior of an important property. The precise statement is:

\begin{proposition}
\label{linear_operations}
\begin{enumerate}
\item Let $M$ be a \emph{cone-preserving} linear transformation $\mathbb{R}^p \to \mathbb{R}^q$, such that for all $v \in C(s,\alpha)$ we have $Mv \in C(s',\alpha')$ and let $X' = XM$.  Suppose further that $|Mv|_q \ge c |v|_q$ for all $v$ in $C(s, \alpha)$.  If $(X,Z)$ has the $\ell_q$ sensitivity property with parameters $(s', \alpha', \gamma)$, then $(X',Z)$ has $\ell_q$ sensitivity with parameters $(s, \alpha, c\gamma)$.

\item Let $M$ be a  linear transformation $\mathbb{R}^L \to \mathbb{R}^T$  such that for all $v$, $|Mv|_{\infty} \ge c|v|_{\infty}$. If we transform $Z' = ZM$, and $(X,Z)$ has the $\ell_q$ sensitivity property with lower bound $\gamma$, then $(X,Z')$ has the same property with lower bound $c\gamma$.
\end{enumerate}
\end{proposition}

One can check that normalization  and feature merging on the $X$ matrix are special cases of the first class of ``cone-preserving'' transformations. For normalization, $M$ is the $p \times p$ diagonal matrix of inverses of the lengths of $X$'s columns. Similarly, normalization on the $Z$ matrix is a special case of the second class of transformations. This shows that our definitions include some concrete commonly performed data processing steps.

\section{Discussion}
\label{discussion}

Our work raises further questions about the theoretical foundations of sparse linear models. What is a good condition to have at the core of the theory? The regularity properties discussed in this paper yield statistical performance guarantees for popular methods such as the Lasso and the Dantzig selector. However, they are not efficiently verifiable. In contrast, incoherence can be checked efficiently, but does not guarantee performance up to the optimal rate \citep{buhlmann_statistics_2011}. It may be of interest to investigate if there are intermediate conditions that achieve favorable trade-offs.
\section{Proofs}
\label{proofs}

\subsection{Proof of Theorem \ref{hardness}}
\label{hardnessproof}

The spark of a matrix $X$, denoted $spark(X)$, is the smallest number of linearly dependent columns. Our proof is a polynomial-time reduction from the $\mathsf{NP}$-hard problem of computing the spark of a matrix (see \cite{bandeira_certifying_2013, tillmann_computational_2012} and references therein).

\begin{lemma}
\label{spark}
 Given an $n \times p$ matrix with integer entries $X$, and a sparsity size $0 < s < p$, it is $\mathsf{NP}$-hard to decide if the spark of $X$ is at most $s$.
\end{lemma}

We also need the following technical lemma, which provides bounds on the singular values of matrices with bounded integer entries. For a matrix $X$, we denote by $\|X\|_2$ or $\|X\|$ its operator norm, and by $X_S$ the submatrix of $X$ formed by the columns with indices in $S$.

\begin{lemma}
\label{singular_values}
Let $X$ be an $n \times p$ matrix with integer entries, and denote $M = \max_{i,j}|X_{ij}|$. Then, we have
$  \|X\|_2  \leq 2^{\lceil \log_2(\sqrt{np} M) \rceil}$. Further, if $spark(X)>s$ for some $0 < s < n$, then for subset $S \subset \{1, \ldots, p\}$ with $|S| = s$, we have:
\[
    \lambda_{\min}(X_S^TX_S) \geq 2^{-2n{\lceil \log_2(nM) \rceil}}.
\]
\end{lemma}

\begin{proof}
The first claim follows from:
$
    \|X\|_2 \leq \sqrt{np} \|X\|_{\max}  \leq 2^{\lceil \log_2(\sqrt{np}M) \rceil}.
$

For the second claim, let $X_S$ denote a submatrix of $X$ with an arbitrary index set $S$ of size $s$.  Then $spark(X)> s$ implies that $X_S$ is non-singular.  Since the absolute values of the entries of $X$ lie in $\{0,\ldots, M\}$, the entries of $X_S^TX_S$ are integers with absolute values between 0 and $nM^2$, namely $\|X_S^T X_S\|_{\max} \leq nM^2$. Moreover, since the non-negative and nonzero determinant of $X_S^T X_S$ is integer, it must be at least 1. Hence,
\begin{align*}
1 \leq  \prod_{i=1}^{s}\lambda_{i}(X_S^TX_S) &\leq \lambda_{\min}(X_S^TX_S) \lambda_{\max}(X_S^TX_S)^{s-1} \\
&\leq   \lambda_{\min}(X_S^TX_S) (s \|X_S^TX_S\|_{\max})^{s-1}.
\end{align*}
Rearranging, we get
\begin{equation*}
\label{lambda_min}
\lambda_{\min}(X_S^TX_S) \geq (snM^2)^{-s+1} \geq (nM)^{-2n} \geq  2^{-2n{\lceil \log_2(nM) \rceil}}.
\end{equation*}
In the middle inequality we have used $s \leq n$.  This is the desired bound.
\end{proof}

For the proof we need the notion of \emph{encoding length}, which is the size in bits of an object. Thus, an integer $M$ has size $\lceil{\log_2(M)}\rceil$ bits. Hence the size of the matrix $X$ is at least $np + \lceil{\log_2(M)}\rceil$: at least one bit for each entry, and $\lceil{\log_2(M)}\rceil$ bits to represent the largest entry. To ensure that the reduction is polynomial-time, we need that the size in bits of the objects involved is polynomial in the size of the input $X$. As usual in computational complexity, the numbers here are rational
\citep{arora_computational_2009}.

\noindent{\bf Proof of Theorem \ref{hardness}}.
It is enough to prove the result for the special case of $X$ with integer entries, since this statement is in fact stronger than the general case, which also includes rational entries. For each property and given sparsity size $s$, we will exhibit parameters $(\alpha, \gamma)$ of polynomial size in bits, such that:
\begin{enumerate}
\item $spark(X) \leq s \mbox{ }\implies \mbox{ } X$ does not obey the regularity property with parameters $(\alpha, \gamma)$,
\item $spark(X) > s  \mbox{ } \implies \mbox{ } X$ obeys the regularity property with parameters $(\alpha, \gamma)$.
\end{enumerate}
Hence, any polynomial-time algorithm for deciding if the regularity property holds for $(X,s,\alpha,\gamma)$, can decide if $spark(X) \le s$ with one call. Here it is crucial that $(\alpha, \gamma)$ are polynomial in the size of $X$, so that the whole reduction is polynomial in $X$. Since deciding $spark(X) \le s$ is $\mathsf{NP}$-hard by Theorem \ref{spark}, this shows the desired $\mathsf{NP}$-hardness of checking the conditions.  Now we provide the required parameters $(\alpha, \gamma)$ for each regularity condition. Similar ideas are used when comparing the conditions.

For the \bf restricted eigenvalue \rm condition, the first claim follows any $\gamma >0$, and any $\alpha > 0$.  To see this,  if the spark of $X$ at most $s$, there is a nonzero $s$-sparse vector $v$ in the kernel of $X$, and $|Xv|_2 = 0 < \gamma|v_S|_2$, where $S$ is any set containing the nonzero coordinates. This $v$ is clearly also in the cone $C(s, \alpha)$, and so $X$ does not obey RE with parameters $(s, \alpha, \gamma)$.

For the second claim, note that if  $spark(X)$ $>$ $s$, then for each index set $S$ of size $s$, the submatrix $X_S$ is non-singular. This implies a nonzero lower bound on the RE constant of $X$. Indeed, consider a vector $v$ in the cone $C(s, \alpha)$, and assume specifically that $\alpha|v_S|_1 \geq |v_{S^c}|_1$.  Using the  identity $Xv = X_Sv_S + X_{S^c}v_{S^c}$, we have
\begin{eqnarray*}
    |Xv|_2 &= & |X_Sv_S + X_{S^c}v_{S^c}|_2 \geq |X_Sv_S|_2 - |X_{S^c}v_{S^c}|_2 \\
 &\geq& \sqrt{\lambda_{\min}(X_S^TX_S)}|v_S|_2 - \|X_{S^c}\|_2|v_{S^c}|_2.
\end{eqnarray*}
Further, since $v$ is in the cone, we have

\begin{equation}
\label{cone_ineq}
    |v_{S^c}|_2 \leq |v_{S^c}|_1 \leq \alpha |v_{S}|_1 \leq \alpha \sqrt{s} |v_S|_2.
\end{equation}

Since $X_S$ is non-degenerate and integer-valued, we can use the bounds from Lemma \ref{singular_values}. Consequently, with $M  = \|X\|_{\max}$,  we obtain

\begin{eqnarray*}
        |Xv|_2 &\geq& |v_S|_2\left(\sqrt{\lambda_{\min}(X_S^TX_S)} - \|X_{S^c}\|\alpha \sqrt{s}\right) \\
    & \geq &
        |v_S|_2\left (2^{-n{\lceil \log_2(nM) \rceil}} - 2^{\lceil \log_2(\sqrt{np} M) \rceil}\alpha \sqrt{s}\right).
\end{eqnarray*}

By choosing, say, $\alpha = 2^{-2n\lceil \log_2(npM) \rceil}$, $\gamma = 2^{-2n\lceil \log_2(npM) \rceil}$, we easily conclude after some computations that $|Xv|_2 \geq \gamma |v_S|_2$. Moreover, the size in bits of the parameters is polynomially related to that of $X$. Indeed, the size in bits of both parameters is $2n\lceil \log_2(npM) \rceil$, and the size of $X$ is at least $np + \lceil{\log_2(M)}\rceil$, as discussed before the proof. Note that $2n\lceil \log_2(npM) \rceil$ $\le$ $(np + \lceil{\log_2(M)}\rceil)^2$. This proves the claim.

The argument for the \bf compatibility \rm conditions is identical, and therefore omitted.

Finally, for the \bf $\ell_q$ sensitivity \rm property, we in fact show that the subproblem where $Z = X$ is $\mathsf{NP}$-hard, thus the full problem is also clearly $\mathsf{NP}$-hard.
The first condition is again satisfied for all $\alpha >0 $ and $\gamma > 0$. Indeed, if the spark of $X$ is at most $s$, there is a nonzero $s$-sparse vector $v$ in its kernel, and thus $|X^TXv|_{\infty} = 0$.

For the second condition, we note that
$
    |Xv|_2^2 = v^TX^TXv \leq |v|_1 |X^TXv|_\infty.
$
For $v$ in the cone, $\alpha|v_S|_1 \geq |v_{S^c}|_1$ and hence
\[
    |v|_2 \geq |v_S|_2 \geq \frac{1}{\sqrt{s}} |v_S|_1 \geq \frac{1}{\sqrt{s}(1+\alpha)}|v|_1.
\]
Combination of the last two results gives
\[
    \frac{s|X^TXv|_{\infty}}{n|v|_1} \geq \frac{s|Xv|_2^2}{n|v|_1^2} \geq \frac {1}{n(1+\alpha)^2}\frac{|Xv|_2^2}{|v|_2^2}.
\]
Finally, since $q  \ge 1$, we have $|v|_1 \ge |v|_q$, and as $v$ is in the cone, $|v|_2^2 = |v_S|_2^2 + |v_{S^c}|_2^2 \leq (1 + \alpha^2s)|v_S|_2^2$, by inequality ~\eqref{cone_ineq}. Therefore,
\[
    \frac{s^{1/q}|X^TXv|_{\infty}}{n|v|_q} \geq \frac {s^{1/q-1}}{n(1+\alpha)^2(1 + \alpha^2s)}\frac{|Xv|_2^2}{|v_S|_2^2}.
\]

Hence we essentially reduced to restricted eigenvalues. From the proof of that case,  the choice  $\alpha = 2^{-2n\lceil \log_2(npM) \rceil}$ gives
$
     |Xv|_2/|v_S|_2 \geq  2^{-2n\lceil \log_2(npM) \rceil}.
$
Hence for this $\alpha$ we also have
$
    s^{1/q}|X^TXv|_{\infty}/(n|v|_2) \geq 2^{-5(n+1)\lceil \log_2(npM) \rceil},
$
where we have applied a number of coarse bounds. Thus $X$ obeys the $\ell_q$ sensitivity property with the parameters $\alpha = 2^{-2n\lceil \log_2(npM) \rceil}$ and  $\gamma = 2^{-5n\lceil \log_2(npM) \rceil}$. As in the previous case, the size in bits of these parameters are polynomial in the size in bits of $X$. This proves the correctness of the reduction for, and completes the proof.\\

\subsection{Proof of Theorem \ref{population_sample}}
\label{proof_popsample}

We first establish some large deviation inequalities for random inner products, then finish the proofs directly by a union bound. We discuss the three probabilistic settings one by one.

\subsubsection{Sub-gaussian variables}
\label{pf1}

\begin{lemma} [Deviation of Inner Products for Sub-gaussians]
\label{lem:inner_prod_concentration}
 Let $X$ and $Z$ be zero-mean sub-gaussian random variables, with sub-gaussian norms $\| X \|_{\psi_2}$, $\| Z \|_{\psi_2}$ respectively. Then, given $n$ iid samples of $X$ and $Z$, the sample covariance satisfies the tail bound:
\begin{equation*}
    \mathbb{P}\left(\left|\frac1n \sum_{i=1}^{n} X_i Z_i - \mathbb{E}(XZ) \right| \ge t  \right)  \leq 2 \exp(-cn\min(t/K,t^2/K^2)).
\end{equation*}
where $K := 4\| X \|_{\psi_2}\| Z \|_{\psi_2}$.
\end{lemma}

\begin{proof}
We use the Bernstein-type inequality in Corollary 5.17 from \cite{vershynin_introduction_2010}. Recalling that the sub-exponential norm of a random vector $X$ is $\| X \|_{\psi_1}$ $=$ $\sup_{p \geq 1} p^{-1} \|X\|_p$, we need to bound the sub-exponential norms of $U_i = X_iZ_i - \mathbb{E}(X_iZ_i)$. We show that if $X,Z$ are sub-gaussian, then $XZ$ has sub-exponential norm bounded by
\begin{equation}
    \label{norms_ineq}
	\| XZ \|_{\psi_1} \leq 2\| X \|_{\psi_2}\| Z \|_{\psi_2}.
\end{equation}

Indeed by the Cauchy-Schwartz inequality $\left(\mathbb{E}|XZ|^p\right)^2 \le \mathbb{E}|X|^{2p} \mathbb{E}|Z|^{2p}$, hence
$p^{-1}\left(\mathbb{E}|XZ|^p\right)^{1/p} \le 2 (2p)^{-1/2} \left(\mathbb{E}|X|^{2p}\right)^{1/2p} (2p)^{-1/2} \left(\mathbb{E}|Z|^{2p}\right)^{1/2p}$.
Taking the supremum over $p \ge 1/2$ leads to \eqref{norms_ineq}.

The $U_i$ are iid random variables, and their sub-exponential norm is bounded as
$
  \| U_i \|_{\psi_1} \le \| X_iZ_i \|_{\psi_1}  + |\mathbb{E}XZ| \leq 2\| X \|_{\psi_2}\| Z \|_{\psi_2} + \left(\mathbb{E}X^2 \mathbb{E}Z^2\right)^{1/2}$. Further, by definition $\left(\mathbb{E}X^2\right)^{1/2} \le \sqrt{2} \| X \|_{\psi_2}$, hence the sub-exponential norm is at most $
  \| U_i \|_{\psi_1} \leq 4\| X \|_{\psi_2}\| Z \|_{\psi_2}$.
The main result then follows by a direct application of Bernstein's inequality, see Corollary 5.17 from \cite{vershynin_introduction_2010}.
\end{proof}

With these preparations, we now prove Theorem \ref{population_sample} for the sub-gaussian case. By a union bound over the $Lp$ entries of the matrix $\Psi - \hat\Psi$
$$
     P(\|\Psi - \hat\Psi\|_{\max}  \ge t)  \leq \sum_{i,j} P(|\Psi_{i,j} - \hat\Psi_{i,j}|  \ge t) \le Lp \max_{i,j} P(|\Psi_{i,j} - \hat\Psi_{i,j}| \ge t).
$$
By Lemma \ref{lem:inner_prod_concentration} each probability is bounded by a term of the form $2\exp(-c n$ $  \min(t/K,t^2/K^2))$, where $K$ varies with $i,j$. The largest of these bounds corresponds to the largest of the $K$-s. Hence the $K$ in the largest term is  $4\max_{i,j}$ $ \|X_i\|_{\Psi_2}\|Z_j\|_{\Psi_2}$. By the definition of sub-gaussian norm, this is at most $4\|\underline{X}\|_{\Psi_2}$  $\|\underline{Z}\|_{\Psi_2}$, where the $\underline{X}$ and $\underline{Z}$ are now $p$ and $L$-dimensional vectors, respectively. Therefore we have the uniform bound
\begin{equation}
\label{uniform_bound}
	P(\|\Psi - \hat\Psi\|_{\max}  \ge t) \le 2Lp\exp(-cn \min(t/K,t^2/K^2))
\end{equation}
with $K = 4\|\underline{X}\|_{\Psi_2} \|\underline{Z}\|_{\Psi_2}$.

We choose $t$ such that $(a+1)\log(2Lp)$  = $cnt^2/K^2$, that is  $t$ = $K [(a+1) \log(2Lp)/cn]^{1/2}$. Since we can assume $(a+1)\log(2Lp)$  $\le$ $cn$ by assumption, the relevant term is the one quadratic in $t$: the total probability of error is $(2Lp)^{-a}$. From now on, we will work on the high-probability event that $\|\Psi - \hat\Psi\|_{\max}  \le t$.

For any vector $v$, $
    \left|\Psi v\right|_\infty - \left|\hat\Psi v\right|_\infty \leq \left|(\Psi - \hat\Psi) v\right|_\infty  \leq  \|\Psi - \hat\Psi\|_{\max} |v|_1 \le t |v|_1$. With high probability it holds uniformly for all $v$ that:
\begin{equation}
\label{eq:rsp_sample_pop}
    \left|\hat\Psi v\right|_\infty \geq \left|\Psi v\right|_\infty - R \sqrt{\frac{\log(2pL)}{n}} |v|_1
\end{equation}
for the constant $R$ = $\sqrt{K^2(a+1)/c}$.

For vectors $v$ in $C(s,\alpha)$, we bound the $\ell_1$ norm by the $\ell_q$ norm, $q \geq 1$, in the usual way, to get a term depending on $s$ rather than on all $p$ coordinates:
\begin{equation}
\label{l1_lq_cone}
    |v|_1 \leq (1 + \alpha)|v_S|_1 \leq (1 + \alpha) s^{1-1/q}|v_S|_q \leq (1 + \alpha) s^{1-1/q}|v|_q.
\end{equation}
Introducing this into \eqref{eq:rsp_sample_pop} gives with high probability over all $v \in C(s,\alpha)$:
\begin{equation*}
    \frac{s^{1/q}\left|\hat\Psi v\right|_\infty}{|v|_q} \geq \frac{s^{1/q}\left|\Psi v\right|_\infty}{|v|_q} - R (1 + \alpha) s \sqrt{\frac{\log(2pL)}{n}} .
\end{equation*}

If we choose $n$ such that
$
 n \ge K^2(1+a)(1+\alpha)^2s^2\log(2pL)/(c\delta^2),
$
then the second term will be at most $\delta$. Further since $\Psi$ obeys the $\ell_q$ sensitivity assumption, the first term will be at least $\gamma$. This shows that $\hat\Psi$ satisfies the the $\ell_q$ sensitivity assumption with constant $\gamma - \delta$ with high probability, and finishes the proof. To summarize, it suffices if the sample size is at least
\begin{equation}
 	\label{eq:sample_size_requirement}
	 n \ge \frac{\log(2pL)(a+1)}{c}\max{\left(1,\frac{K^2(1+\alpha)^2}{\delta^2}s^2\right)}.
\end{equation}

\subsubsection{Bounded variables}
\label{pf2}

If the components of the vectors $X,Z$ are bounded, then essentially the same proof goes through.   The sub-exponential norm of $X_iZ_j - \mathbb{E}(X_iZ_j)$ is bounded -- by a different argument --  because $|X_iZ_j - \mathbb{E}(X_iZ_j)|$ $\le$ $2 C_xC_z$, hence $\|X_iZ_j - \mathbb{E}(X_iZ_j)\|_{\Psi_1}$ $\le$ $2 C_xC_z$. Hence Lemma \ref{lem:inner_prod_concentration} holds with the same proof, where now the value of $K := 2 C_xC_z$ is different. The rest of the proof only relies on Lemma \ref{lem:inner_prod_concentration}, so it goes through unchanged. 

\subsubsection{Variables with bounded moments}
\label{pf3}

For variates with bounded moments, we also need a large deviation inequality for inner products. The general flow of the argument is classical, and relies on the Markov inequality and a moment-of-sum computation (e.g., \cite{petrov_limit_1995}).  The result is a generalization of a lemma used in covariance matrix estimation \citep{ravikumar_high-dimensional_2011}, and our proof is shorter.

\begin{lemma} [Deviation for Bounded Moments - Khintchine-Rosenthal]
\label{lem:inner_prod_poly}
 Let $X$ and $Z$ be zero-mean random variables, and $r$ a positive integer, such that $\mathbb{E}X^{4r} = C_x $, $\mathbb{E}Z^{4r} = C_z$. Given $n$ iid samples from $X$ and $Z$, the sample covariance satisfies the tail bound:
\begin{equation*}
    \mathbb{P}\left(\left|\frac1n \sum_{i=1}^{n} X_i Z_i - \mathbb{E}(XZ) \right| \ge t  \right)  \leq \frac{2^{2r}r^{2r}\sqrt{C_xC_z}}{t^{2r}n^r}.
\end{equation*}

\end{lemma}

\begin{proof}
Let $Y_i = X_iZ_i - \mathbb{E}XZ$, and $k=2r$. By the Markov inequality, we have
\begin{equation*}
    \mathbb{P}\left(\left|\frac1n \sum_{i=1}^{n} Y_i \right| \ge t  \right)  \leq \frac{\mathbb{E}\left|\sum_{i=1}^{n} Y_i \right|^k}{(tn)^k}.
\end{equation*}

We now bound the $k$-th moment of the sum $\sum_{i=1}^{n} Y_i$ using a type of classical argument, often referred to as Khintchine's or Rosenthal's inequality. We can write, recalling that $k = 2r$ is even,
\begin{equation} \label{sum-jq}
    \mathbb{E}\left|\sum_{i=1}^{n} Y_i \right|^{k} = \sum_{i_1, i_2, \ldots, i_k \in \{1, \ldots, n\}}\mathbb{E(}Y_{i_1}Y_{i_2} \ldots Y_{i_k})
\end{equation}
By independence of $Y_i$, we have $ \mathbb{E}(Y_1^{a_1}Y_2^{a_2} \ldots Y_n^{a_n}) = \mathbb{E}Y_1^{a_1}\mathbb{E}Y_2^{a_2} \ldots \mathbb{E}Y_n^{a_n}$. As $\mathbb{E}Y_i = 0$, the summands for which there is a $Y_i$ singleton vanish. For the remaining  terms, we bound by Jensen's inequality $(\mathbb{E}|Y|^{r_1})^{1/r_1} \leq (\mathbb{E}|Y|^{r_2})^{1/r_2}$ for $0 \leq r_1\leq r_2$. So each term is bounded by $
    (\mathbb{E}|Y|^{k})^{a_1/k} \ldots (\mathbb{E}|Y|^{k})^{a_n/k} =  \mathbb{E}|Y|^{k}$.

Hence, each nonzero term in \eqref{sum-jq} is uniformly bounded. We count the number of sequences of non-negative integers $(a_1, \ldots, a_n)$ that sum to $k$, and such that if some $a_i >0$, then $a_i \geq 2$. Thus, there are at most $k/2 = r$ nonzero elements. This shows that the number of such sequences is not more than the number of ways to choose $r$ places out of $n$, multiplied by the number of ways to distribute $2r$ elements among those places, which can be bounded by  $
    \binom{n}{r}r^{2r} \leq n^r r^{2r}$.
Thus, we have proved that
$
\mathbb{E}\left|\sum_{i=1}^{n} Y_i \right|^{2r} \leq  n^r r^{2r} \mathbb{E}|Y|^{2r}.
$

We can make this even more explicit by the Minkowski and Jensen inequalities:  $
	\mathbb{E}|Y|^{k} = \mathbb{E}|X_iZ_i-\mathbb{E}X_iZ_i|^{k}$ $\le \left( (\mathbb{E}|X_iZ_i|^{k})^{1/k}+\mathbb{E}|X_iZ_i|\right)^{k}$  $\le 2^k \mathbb{E}|X_iZ_i|^{k}$.
Combining this with $\mathbb{E}|X_iZ_i|^{k}$ $\le$ $\sqrt{\mathbb{E}|X_i|^{2k}\mathbb{E}|Z_i|^{2k}}$ $=$ $\sqrt{C_xC_z}$ leads to the desired bound
$
    \mathbb{P}\left(\left|\frac1n \sum_{i=1}^{n} Y_i \right| \ge t  \right)
    \leq 2^{2r}r^{2r}\sqrt{C_xC_z}/(t^{2r}n^r).
$
\end{proof}

To prove Theorem \ref{population_sample}, we  note that by a union bound, the probability that $\|\Psi - \hat\Psi\|_{\max}  \ge t$ is at most
$
Lp 2^{2r}r^{2r}\sqrt{C_xC_z}/(t^{2r}n^r).
$
Since $r$ is fixed, for simplicity of notation, we can denote $C_0^{2r} = 2^{2r}r^{2r}\sqrt{C_xC_z}$. Choosing $t$ $=$ $C_0(Lp)^{1/2r} n^{-1/2 + a/(2r)}$, the above probability is at most $1/n^a$.

The bound $\left|\Psi v\right|_\infty - \left|\hat\Psi v\right|_\infty \leq \left|(\Psi - \hat\Psi) v\right|_\infty  \leq  \|\Psi - \hat\Psi\|_{\max} |v|_1$ holds as before, so we conclude that with probability $1 - 1/n^a$, for all $v \in C(s,\alpha)$:
\begin{equation*}
\frac{s^{1/q}\left|\hat\Psi v\right|_\infty}{|v|_q} \geq \frac{s^{1/q}\left|\Psi v\right|_\infty}{|v|_q} - (1 + \alpha) s t .
\end{equation*}
From the choice of $t$, for sample size at least
$
n^{1 - a/r} \ge C_0^2(1+\alpha)^2(Lp)^{1/r}s^2/(\delta^2),
$
the error term on the left hand side is at most $\delta$, which is what we need.

\subsection{Proof of Theorem \ref{k-comprehensive}}
\label{proof_comprehensive}

To bound the term $|\Psi v|_{\infty}$ in the $\ell_1$ sensitivity, we use the s-comprehensive property. For any $v \in C(s,\alpha)$, by the symmetry of the s-comprehensive  property, we can assume without loss of generality that $|v_1|$ $\ge$ $|v_2|$ $\ge \ldots \ge$ $|v_p|$. Then if $S$ denotes the first $s$ components, $\alpha |v_S|_1 \ge |v_{S^c}|_1$.

Consider the sign pattern of the top $s$ components of $v$: $\varepsilon = (\text{sgn}(v_1),\ldots,\text{sgn}(v_s))$. Since $\Psi$ is s-comprehensive, it has a row $w$ with matching sign pattern. Then we can compute

$$
	\langle w, v \rangle = \sum_{i \in S} |w_i|\text{sgn}(w_i) v_i = \sum_{i \in S} |w_i|\text{sgn}(v_i) v_i =
\sum_{i \in S} |w_i||v_i|.
$$
Hence the inner product is lower bounded by
$
	\min_{i \in S}|w_i| \sum_{i \in S}|v_i| \ge c \sum_{i \in S}|v_i|.
$
Combining the above, we get the desired bound:
$$
	\frac{s|\langle w, v\rangle|}{|v|_1} \ge
	\frac{s c |v_S|_1}{(1+\alpha)|v_S|_1} =
	\frac{c s}{(1+\alpha)}.
$$

\subsection{Proof of claims in Examples \ref{covariance_example_1}, \ref{covariance_example_2}}
\label{proof_example}

We bound the $\ell_1$ sensitivity for the two specific covariance matrices $\Sigma$. For the diagonal matrix in Example \ref{covariance_example_1}, with entries $d_1, \ldots, d_p>0$, we have
$
m = |\Sigma v|_{\infty} = \max( {|d_1v_1|, \ldots, |d_pv_p|}).
$
Then summing $|v_i| \le m/d_i$ for $i$ in any set $S$ with size $s$, we get
$
|v_S|_1 \le m \sum_{i \in S} 1/d_i.
$
To bound this quantity for $v \in C(s,\alpha)$, let $S$ be the subset of dominating coordinates for which $|v_{S^c}|_1 \le \alpha |v_S|_1$. It follows that
$
|v|_1 \le (1 + \alpha) |v_S|_1 \leq  (1 + \alpha)  m \sum_{i \in S} 1/d_i.
$
Therefore
$$
\frac{ s  |\Sigma v|_{\infty}} {|v|_1} \ge \frac {s} { (1+\alpha) \sum_{i \in S} 1/d_i} \geq \frac {1} { (1+\alpha)  s^{-1} \sum_{i=1}^s 1/d_{(i)}},
$$
where  $\{d_{(i)}\}_{i=1}^p$ is the order of $\{d_{i}\}_{i=1}^p$, arranged from the
smallest to the largest.  The harmonic average in the lower bound can be bounded away from zero even several $d_i$-s are of order $O(1/s)$.
For instance if $d_{(1)} = \cdots = d_{(k)} = 1/s$ and $d_{(k+1)} > 1/c$ for some constant $c$ and integer $k < s$, then the $\ell_1$ sensitivity is at least
$
{s  |\Sigma v|_{\infty}}/|v|_1 \ge 1/[(1+\alpha) (k+ (1-k/s)c)],
$
which is bounded away from zero whenever $k$ is bounded. In this setting the smallest eigenvalue of $\Sigma$ is $1/s$, so only the $\ell_1$ sensitivity holds out of all regularity properties.

For the covariance matrix in example \ref{covariance_example_2},
$$
m = |\Sigma v|_{\infty} = \max( {|v_1+\rho v_2|, |v_2 + \rho v_1|, |v_3|, \ldots, |v_p|}).
$$
The coordinate $v_1$ can be bounded as follows:
$$
|v_1|=\left |\frac{1}{1-\rho^2}(v_1+\rho v_2) - \frac{\rho}{1-\rho^2}(\rho v_1+ v_2) \right |
\le m \left( \frac{1}{1-\rho^2} + \frac{\rho}{1-\rho^2} \right)
$$
leading to $|v_1| \le m/(1-\rho)$. Similarly $|v_2| \le m/(1-\rho)$.  Furthermore,
For each $i \not \in \{1,2\}$, we have $|v_i| \le m$.  Thus, for any set $S$,
$
|v_S|_1 \le m [ 2/(1-\rho) + s-2 ].
$
For any $v \in C(s,\alpha)$,
$
    |v|_1 \le (1 + \alpha) |v_S|_1 \leq (1+\alpha) m \left ( 2/(1-\rho) + s-2 \right ).
$
leading to a lower bound on the $\ell_1$ sensitivity:
$$
\frac{ s  |\Sigma v|_{\infty}} {|v|_1} \ge \frac{s}{(1+\alpha)(2/(1-\rho) + s-2)}.
$$
If $1- \rho = 1/s$, this bound is at least $1/3(1+\alpha)$, showing that $\ell_1$ sensitivity holds. However, the smallest eigenvalue is also $1-\rho = 1/s$, so the other regularity  properties (restricted eigenvalue, compatibility), fail to hold as $s \to \infty$.

\subsection{Proof of Proposition \ref{linear_operations} }
\label{operation_proofs}

For the first claim, note $(Z')^TX'v = Z^TX(Mv)$. If $v$ is any vector in the cone $C(s,\alpha)$, we have $Mv \in C(s', \alpha')$ by the cone-preserving property. Hence by the $\ell_q$ sensitivity of $X,Z$
$
s^{1/q} |n^{-1}Z^TX(Mv)|_{\infty}/|Mv|_q \ge \gamma.
$
Multiplying this by  $|Mv|_q \ge c |v|_q$ yields the $\ell_q$ sensitivity for $X', Z$.

For the second claim, we write  $(Z')^TX'v = MZ^TXv$. By the $\ell_q$ sensitivity of $X,Z$, for all $v \in C(s, \alpha)$,
$
s^{1/q} |n^{-1}Z^TXv|_{\infty}/|v|_q \ge \gamma.
$
Multiplying this by $n^{-1}|MZ^TXv|_{\infty} \ge cn^{-1}|Z^TXv| _{\infty}$ finishes the proof.

\bibliographystyle{spmpsci}      
\bibliography{senior.thesis}  

\end{document}